%% file: inftygraph9.tex
\documentclass[12pt,a4paper]{amsart}
\usepackage[latin1]{inputenc}
\usepackage{graphicx}
\usepackage{amssymb, amsmath}
\usepackage{geometry}
\usepackage{enumerate}
\usepackage[colorlinks=true,linkcolor=blue,urlcolor=blue,citecolor=blue]{hyperref}
\newtheorem{theorem}{Theorem}[section]
\newtheorem{definition}[theorem]{Definition}
\newtheorem{lemma}[theorem]{Lemma}

\newtheorem{proposition}[theorem]{Proposition}

\theoremstyle{definition}

\newtheorem{remark}[theorem]{Remark}

\newcommand{\1}{\mathbf{1}}
\newcommand{\N}{\mathbb{N}}

\DeclareFontFamily{U}{wncy}{}
\DeclareFontShape{U}{wncy}{m}{n}{<->wncyr10}{}
\DeclareSymbolFont{mcy}{U}{wncy}{m}{n}
\DeclareMathSymbol{\Sh}{\mathord}{mcy}{"58}

\begin{document}
\title[Geometric properties of infinite graphs]{Geometric properties of infinite graphs and the Hardy-Littlewood maximal operator}
\author{Javier Soria}
\address{Department of  Applied Mathematics and Analysis, University of Barcelona, Gran Via 585, E-08007 Barcelona, Spain.}
\email{soria@ub.edu}

\author{Pedro Tradacete}
\address{Mathematics Department, Universidad Carlos III de Madrid, E-28911 Legan\'es, Ma\-drid, Spain.}
\email{ptradace@math.uc3m.es}

\thanks{The first author has  been partially supported by the Spanish Government grants MTM2013-40985, and the Catalan Autonomous Government grant 2014SGR289. The second author has   been partially supported by the Spanish Government grants MTM2013-40985 and MTM2012-31286, and also Grupo UCM 910346.}


\begin{abstract}
We study different geometric properties on infinite graphs, related to the weak-type boundedness of the Hardy-Littlewood maximal averaging operator. In particular, we analyze the connections between the doubling condition, having finite dilation and overlapping indices, uniformly bounded degree, the equidistant comparison property and the weak-type boundedness of the centered Hardy-Littlewood maximal operator. Several non-trivial examples of infinite graphs are given to illustrate the differences among these properties.
\end{abstract}

\maketitle

\thispagestyle{empty}

\section{Introduction}

Let  $G=(V_G,E_G)$ be a simple and connected graph, with  vertices having finite degrees $1\le d_x<\infty$ (conditions that we will always assume from now on), where $V_G$ is the set of vertices and $E_G$ is the set of edges between them. Unless otherwise stated, we  will only  consider the case of infinite graphs. 

For a function $f:V_G\rightarrow \mathbb R$,  the (centered) Hardy-Littlewood maximal operator is defined as
$$
M_G f(x)=\sup_{r\ge 0} \frac{1}{|B(x,r)|}\sum_{y\in B(x,r)} |f(y)|.
$$
Here, $B(x,r)$ denotes the ball of center $x$ and radius $r$ on the graph, equipped with the metric $d_G$ induced by the edges in $E_G$. That is, given $x,y\in V_G$ the distance $d_G(x,y)$ is the number of edges in a shortest path connecting $x$ and $y$, and
$$
B(x,r)=\{y\in V_E:d_G(x,y)\le r\}.
$$
Similarly, we define the sphere
$$
S(x,r)=\{y\in V_E:d_G(x,y)= r\}.
$$
Observe that  $B(x,r)=B(x,[r])$, where $[r]$ denotes the integer part of $r$, $B(x,r)=\{x\}$, if $0\le r<1$ and $B(x,r)=\{x\}\cup N_G(x)$, if $1\le r<2$, where $N_G(x)=S(x,1)$ is the set of neighbors of $x$; i.e., all vertices adjacent to $x$. Also, given a finite set $A\subset V_G$ we denote its cardinality by $|A|$. Thus,  $d_x=|N_G(x)|$. It is clear that, since the vertices' degrees are finite, all balls are finite sets. We refer to \cite{Bol, BoMu} for standard notations and terminology on graphs.

Since the distance $d_G$ introduced above only takes natural numbers as values, the radius $r\ge0$ considered in the definition of the Hardy-Littlewood maximal operator can be taken to be a natural number $r\in\mathbb N=\{0,1,2,\dots\}$:
$$
M_G f(x)=\sup_{r\in \N} \frac{1}{|B(x,r)|}\sum_{y\in B(x,r)} |f(y)|.
$$

The relevance of the Hardy-Littlewood maximal operator is well-known. It is a key tool in the study of averages of functions, and in particular in their limiting behavior, since it yields applications to differentiation theorems, among others. We refer the reader to the book by E. Stein \cite{Stein} as a reference for maximal operators and their use in analysis.

Motivated by the results of A. Naor and T. Tao \cite{NaorTao}, for the case of the infinite $k$-regular tree, as well as our previous work \cite{Sor-Tra_graphs}, for finite graphs, our aim in this paper is to study weak-type $(1,1)$ boundedness for $M_G$ in terms of the geometry of the graph; that is, estimates of the form
\begin{equation}\label{wtp11}
|\left\{ x \in V_G: M_G f(x) > \lambda \right\}|\le 
\frac{C_G}{\lambda} \| f \|_{L_1(G)},
\end{equation}
for all $f \in L_1(G)$ and $\lambda > 0$, where $C_G$ is a constant depending only on $G$. This inequality is usually written as
$$
M_G:L_1(G)\rightarrow L_{1, \infty}(G),
$$
where 
$$
\|f\|_{1,\infty}=\sup_{\lambda>0}\lambda|\left\{ x \in V_G:   |f(x)| > \lambda \right\}|,
$$
and it is denoted as the weak-type (1,1) boundedness of $M_G$. 

The Hardy-Littlewood maximal operator in metric measure spaces has been mainly considered in a doubling setting (cf. \cite{heinonen}). In non-doubling spaces a natural modification of the maximal operator has also been studied (see \cite{NTV,Sawano,Stempak}). There has also been several attempts to analyze discrete versions of the Hardy-Littlewood maximal operator \cite{AK, CH,MSW,SW}.

For our purpose, we will consider in Section~\ref{geomprop} dilation and overlapping indices, $\mathcal{D}_k(G)$ and $\mathcal{O}(G)$, and the \textit{ECP} property, related to estimate \eqref{wtp11}, and show some preliminary results, in particular the boundedness on the Dirac deltas (Proposition~\ref{diracdeltas}). In Section~\ref{allcases}, we will study the different relationships among all these tools, and we will compute the explicit values of these indices for some relevant examples of infinite graphs, allowing us to show concrete counterexamples to the necessity or sufficiency of these properties (we summarize all this in  Table~\ref{table1}). Finally, in Section~\ref{expander},  we follow the ideas used in \cite{NaorTao} to analyze the case of the closely related spherical maximal function (Theorem~\ref{thm:weak11}).

\section{Geometrical properties}\label{geomprop}

In order to study the analogous results of the classical weak-type $(1,1)$ bounds for the Hardy-Littlewood operator on $\mathbb R^n$,  we introduced in \cite{Sor-Tra_graphs} two numbers associated to a graph $G$: the dilation and the overlapping indices. The dilation index of a graph is related to the so called doubling condition, and measures the growth of the number of vertices in a ball when its radius is enlarged in a fixed proportion.

\begin{definition}
Given a graph $G$, for every $k\in\mathbb N$, $k\geq2$, we define the $k$-dilation index as
$$
\mathcal{D}_k(G)=\sup\bigg\{\frac{|B(x,kr)|}{|B(x,r)|}:x\in V_G,\, r\in\mathbb N\bigg\}.
$$
\end{definition}

\begin{remark} Given $k\geq2$, one can easily compute the $k$-dilation index of some relevant finite graphs: For $n\in\mathbb N$, $n\ge2$, let $K_n$ be the complete graph, and $S_n$ the star graph with $n$ vertices (i.e., $S_n$ is a graph with one vertex of degree $n-1$, and $n-1$ leaves, or vertices of degree 1). Then we have
$$
\mathcal{D}_k(K_n)=1 \qquad\text{and}\qquad
\mathcal{D}_k(S_n)=\frac{n}{2}.
$$
Also, for $L_n$, the linear tree with $n$ vertices, we have that $\mathcal D_k(L_n)< k$, for all $n\in \mathbb {N}$, and $\lim_{n\rightarrow\infty}\mathcal{D}_k(L_n)=k$. For the case of infinite graphs, see Section~\ref{allcases}.
\end{remark}

It is easy to see that if $k,k'\in\mathbb N$, with $2\le k<k'$, then
$$
\mathcal D_k(G)\leq\mathcal D_{k'}(G)\leq\mathcal D_{k}(G)^{\big[\frac{\log k'}{\log k}\big]+1}.
$$
In particular, if for a certain graph $G$, there is $k_0\geq2$ such that $\mathcal D_{k_0}(G)<\infty$, then for every $k\geq2$, $\mathcal D_k(G)<\infty$.

The dilation index is very much related to the following property: A metric measure space $(X,d_X,\mu)$ satisfies the doubling condition (cf. \cite{heinonen}) if there is $K>0$ such that, for every $r>0$ and every $x\in X$, we have that 
\begin{equation}\label{doubmu}
\mu(B(x,2r))\leq K\mu(B(x,r)).
\end{equation}
However, the fact that we only need to consider $r\in\mathbb N$ makes the $k$-dilation index more suitable, and less restrictive, as the following result shows. Before, let us recall that the maximum degree of a graph $G$ is defined as 
$$
\Delta_G=\sup\{d_x:x\in V_G\}.
$$

\begin{proposition}\label{deldildou}
Let $G$ be a graph and denote 
$$
K(G)= \sup\bigg\{\frac{|B(x,2r)|}{|B(x,r)|}:x\in V_G,\, r>0\bigg\}.
$$
Then,
$$
\max\big\{\mathcal{D}_2(G),1+\Delta_G\big\}\leq K(G)\le  \mathcal D_2(G)(1+\Delta_G).
$$
In particular, a graph $G$ satisfies the doubling condition \eqref{doubmu} if and only if both $\Delta_G$ and $\mathcal{D}_2(G)$ are finite.
\end{proposition}

\begin{proof}
It is clear that $K(G)\ge \mathcal{D}_2(G)$. Also, for any $x\in V_G$ we have that
$$
K(G)\ge\frac{|B(x,1)|}{|B(x,1/2)|}=1+d_x.
$$
Hence, $K(G)\ge 1+\Delta_G$. 

Conversely, for $0<r<1$ we have 
$$
|B(x,2r)|\le|B(x,1)|\leq1+\Delta_G\le\mathcal D_2(G)(1+\Delta_G).
$$ 
While for $r\geq1$ we have 
$$
B(x,[r]+1)= \bigcup_{y\in B(x,[r])}S(y,1).
$$
Hence,  
$$
|B(x,2r)|\le |B(x,2([r]+1)|\le\mathcal D_2(G)|B(x,[r]+1)|\le\mathcal D_2(G)|B(x,r)|\Delta_G.
$$ 
In conclusion, in any case we get that $K(G)\le \mathcal D_2(G)(1+\Delta_G)$.
\end{proof}

\medskip

We will see in Proposition~\ref{onodel}   an example of a graph $G$ for which  $\Delta_G=\infty$   but $\mathcal D_k(G)$ is finite, and Proposition~\ref{trkdo} provides another graph where $\Delta_G$ is finite and $\mathcal D_k(G)=\infty$  (hence, in both cases $K(G)=\infty$). This shows that the conditions on Proposition~\ref{deldildou} are actually independent.

\medskip

We now recall the overlapping index of a graph $G$, which represents the smallest number of balls that necessarily overlap in any covering of $G$.

\begin{definition}
Given a graph $G$, we define its overlapping index as
\begin{align*}
\mathcal{O}(G)=\min\bigg\{m\in\mathbb N: \ & \forall \{B_j\}_{j\in J},\,  B_j\text{ a ball in } G,\,\exists\, I\subset J,\\
&\qquad\qquad  \bigcup_{j\in J}B_j=\bigcup_{i\in I}B_i \text{ and }\sum_{i\in I}\chi_{B_i}\leq m\bigg\}.
\end{align*}
\end{definition}

The overlapping index of the following families of finite graphs can be easily computed:

\begin{align*}
\mathcal{O}(K_n)&=1,\,   n\ge1; \quad
&\mathcal{O}(S_n)&= n-1,\, n\geq2;\\
\mathcal{O}(L_n)&=\left\{
\begin{array}{ccc}
 1, &   & 1\le n\leq2,  \\
 2, &   & n\geq 3;
\end{array}
\right.
\quad
&\mathcal{O}(C_n)&=\left\{
\begin{array}{ccc}
 1, &   & n=3, \\
 2, &   & n\geq 4,
\end{array}
\right.
\end{align*}
where $C_n$ is the cycle with $n$ vertices. Examples for infinite graphs will be given in Section~\ref{allcases}.
\medskip

\begin{remark}
The definition of the overlapping index suggests some connection with dimension theory. Recall that for a metric space $(X,d)$, its asymptotic dimension $\mathrm{asdim}(X)$ is the smallest $n\in\mathbb N$ such that, for every $r>0$, there is $D(r)>0$ and a covering of $X$ with sets of diameter smaller than $D(r)$, such that every ball in $X$ of radius $r$ intersects at most $n+1$ members of the covering (cf. \cite[page 29]{Gromov}). 

Note that if a graph $G$ has finite overlapping index, then for every $r>0$ we can always find a collection of balls of radius $r$, $(B_i)_{i\in I}$ such that 
$$
\sum_{i\in I}\chi_{B_i}\leq \mathcal O(G).
$$ 
Hence, we always have
$$
\mathrm{asdim}(G)\leq \mathcal O(G)-1.
$$
Analogous estimates hold for other notions of dimension considered in the metric setting (such as Assouad-Nagata dimension, cf. \cite{nagata}.) Note however, that for the $k$-regular tree $T_k$ we have $\mathrm{asdim}(T_k)=1$, while $\mathcal{O}(T_k)=\infty$, for $k\geq3$ (see Proposition~\ref{trkdo}).
\end{remark}

The dilation and overlapping indices were used in \cite{Sor-Tra_graphs} to obtain an upper bound for the weak-type $(1,1)$ norm of the maximal operator of a finite graph. The proof also works for infinite graphs, yielding the estimate
\begin{equation}\label{eq:weak11}
\|M_G\|_{1,\infty}\leq\min\big\{\mathcal{D}_3(G),\mathcal{O}(G)\big\}.
\end{equation}

We will introduce next a new property concerning the size of balls on equidistant points, and we will also study its relation with the boundedness of $M_G$.
\begin{definition}
A graph $G$ has the equidistant comparability property (ECP, in short) if there is a constant $C\geq1$ such that, for every $x,y\in V_G$
$$
\frac1C|B(x,d(x,y))| \leq |B(y,d(x,y))| \leq C |B(x,d(x,y))|.
$$
In this case, we define  the ECP constant of $G$ as 
$$
\mathcal C(G)=\sup\bigg\{\frac{|B(x,d(x,y))|}{|B(y,d(x,y))|}:x,y\in V_G\bigg\}.
$$
\end{definition}

\begin{remark}\label{doubECP}
Every vertex-transitive graph $G$ (i.e., for every $x,y\in V_G$ there is an automorphism of $G$ mapping $x$ to $y$) has the \textit{ECP}. In particular, every Cayley graph has the \textit{ECP}. Also if every pair of balls in the graph with equal radius are comparable, then $G$ has the \textit{ECP}.  For instance, the infinite $k$-regular tree $T_k$ satisfies that $|B(x,r)|=|B(y,r)|$, for every $x,y\in V$ and $r\ge0$, and hence the \textit{ECP} holds with $\mathcal C(T_k)=1$ (for more information, see Proposition~\ref{trkdo}). 

Observe also that, since for any pair of vertices $x,y$ in a graph $G$, we have that $B(x,d(x,y))\subset B(y,2d(x,y))$, then  the dilation condition implies the \textit{ECP}. In fact, 
$$
\mathcal C(G)\le \mathcal {D}_2(G).
$$ 
\end{remark}

Other important estimates relating these indices are given in the following result.

\begin{proposition}
Suppose $G$ is a graph with the \textit{ECP}  and $\mathcal O (G)<\infty$. Then, 
$$
\mathcal{D}_2(G)\le1+\mathcal O(G) \mathcal C(G).
$$
\end{proposition}

\begin{proof}
Fix $x\in V_G$ and $r\in\mathbb N$. Since $\mathcal O (G)<\infty$, if we consider the collection of balls $\{ B(y,r): \,y\in S(x,r)\}$, we can find a subset $F\subset S(x,r)$ such that
$$
\bigcup_{y\in S(x,r)} B(y,r)=\bigcup_{y\in F} B(y,r),
$$
with the additional property that 
$$
\sum_{y\in F}\chi_{B(y,r)}\leq\mathcal O(G)
$$ 
(observe that since $G$ is an infinite graph, then $S(x,r)\neq\emptyset$).  Note that $x\in B(y,r)$, for every $y\in F\subset S(x,r)$, and  we  have that
$$
|F|=\sum_{y\in F}\chi_{B(y,r)}(x)\leq\mathcal O(G).
$$

Now, since $G$ has the \textit{ECP},  for every $y\in S(x,r)$ we have $|B(y,r)|\leq \mathcal C(G)|B(x,r)|$. Hence, since 
$$
B(x,2r)=\bigcup_{y\in S(x,r)} B(y,r)\cup B(x,r),
$$ 
we get
\begin{align*}
|B(x,2r)|&\le\bigg|\bigcup_{y\in F} B(y,r)\bigg|+|B(x,r)|\leq\sum_{y\in F}|B(y,r)|+|B(x,r)|\\
&\leq\big(1+\mathcal O(G) \mathcal C(G)\big)|B(x,r)|,
\end{align*}
which proves the result.
\end{proof}

\begin{remark}
Arguing in a similar manner one can show that if a graph $G$ has the \textit{ECP} and   $\Delta_G<\infty$, then $G$ satisfies the so-called \textit{local doubling condition}; i.e., for every $r\in\mathbb N$, there exists $D_r<\infty$ such that for every $x\in V_G$, $|B(x,2r)|\leq D_r|B(x,r)|$ (cf. \cite{CMM}). Indeed, since 
$$
B(x,2r)=\bigcup_{y\in S(x,r)} B(y,r)\cup B(x,r),
$$ 
we have
\begin{align*}
|B(x,2r)|&\leq\sum_{y\in S(x,r)}|B(y,r)|+|B(x,r)|\leq |S(x,r)|\mathcal C(G)|B(x,r)|+|B(x,r)|\\
&\leq\big(1+\Delta_G^r\mathcal C(G)\big)|B(x,r)|.
\end{align*}
\end{remark}

The following result shows that \textit{ECP} implies the   weak-type (1,1) boundedness of $M_G$ on the Dirac deltas (this should be compared to \cite[Lemma 2.5]{Sor-Tra_graphs}):

\begin{proposition}\label{diracdeltas}
If $G$ has ECP, then
$$ 
\sup_{x\in V_G}\|M_G\delta_x\|_{1,\infty}\leq \mathcal C(G).
$$
\end{proposition}

\begin{proof}
For any $x\in V_G$ and $0<\lambda<1$, if we set 
$$
r_x=\max\Big\{d(x,y): |B(x,d(x,y))|<\frac{\mathcal C(G)}{\lambda}\Big\},
$$ 
then we have that $B(x,r_x)=\left\{y\in V_G:|B(x,d(x,y))|<\frac{\mathcal C(G)}{\lambda}\right\}$ and 
\begin{align*}
\left|\left\{y\in V_G:|M_G\delta_x (y)|>\lambda\right\}\right|&=\left|\left\{y\in V_G:|B(y,d(x,y))|<\frac1\lambda\right\}\right|\\
&\leq \left|\left\{y\in V_G:|B(x,d(x,y))|<\frac{\mathcal C(G)}{\lambda}\right\}\right|\\
&=|B(x,r_x)| <\frac {\mathcal C(G)}\lambda.
\end{align*}
Therefore,
$$
\|M_G\delta_x\|_{1,\infty}=\sup_{\lambda>0}\lambda\left|\left\{y\in V_G:|M_G\delta_x (y)|>\lambda\right\}\right|\leq \mathcal C(G).
$$
\end{proof}

\section{Examples of graphs}\label{allcases}
The purpose of this section is to study the relations between the geometric conditions introduced in Section~\ref{geomprop} and the weak-type estimates for $M_G$, illustrating, by means of several examples, the differences between boundedness of $\|M_G\|_{1,\infty}$, \textit{ECP}, dilation and overlapping conditions, and the maximum degree.

\medskip

\subsection{The direct sum of complete graphs: $\oplus K_n$}
Let 
$$
V_{\oplus K_n}=\{(m,n)\in \mathbb N^2: m\geq2,\,1\leq n\leq m-1\}.
$$
We set two vertices $(m_1,n_1)$, $(m_2,n_2)$ in $V_{\oplus K_n}$ to be adjacent as follows
$$
(m_1,n_1)\sim_{\oplus K_n} (m_2,n_2) \Leftrightarrow \left\{
                                          \begin{array}{c}
                                            m_1=m_2+1,\, n_1=1;\, \mathrm{or},\\
                                            m_2=m_1+1,\, n_2=1;\, \mathrm{or},\\
                                            m_1=m_2.
                                          \end{array}
                                        \right.
$$

\begin{center}
\begin{figure}[hb]
\input{Kns.tex} \\[-1.5cm]
\caption{The graph $\oplus K_n$.}
\end{figure}
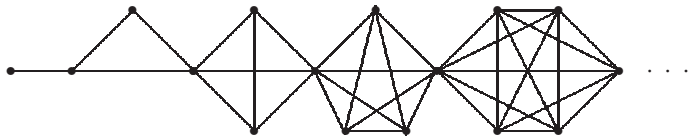
\end{center}

\bigbreak

\begin{proposition}\label{onodel} The following properties hold for the graph $\oplus K_n$:
\begin{enumerate}[\rm (i)]
\item $\Delta_{\oplus K_n}=\infty$.

\item
$\mathcal O(\oplus K_n)=2$.

\item $\mathcal D_2(\oplus K_n)\le 48$. 

\item$\mathcal C(\oplus K_n)\le 48$. 

\item $M_{\oplus K_n}:L_1\rightarrow L_{1,\infty}$ is bounded, with norm not exceeding $2$.
\end{enumerate}
\end{proposition}

\begin{proof}
It is clear that for $m\geq 2$, $d_{(m,1)}=(m-1)+(m-2)=2m-3$, and hence $\Delta_{\oplus K_n}=\infty$. 

To describe the balls of this graph, we are going to think of $\oplus K_n$ as the union of all complete graphs with a common cut vertex $v_j$ between $K_j$ and $K_{j+1}$, $j=1,2,\dots$ (we will write $v_j=(j+1,1)\in K_j\cap K_{j+1}$). All the other vertices $v$ in $K_j$ will be referred as interior vertices ($v\in \mathring{K_j}$). Now, it is easy to see that, if $r=1,2,\dots$, and we define $[m]_1=\max\{m,1\}$, then
\begin{equation}\label{ballcompl}
B(v,r)=\begin{cases}
  \displaystyle\bigcup_{l=[j-r]_1}^{j+r}  K_l, & \text{ if  }  v\in K_j\cap K_{j+1}, \\[.8cm]
\displaystyle   \bigcup_{l=[j-r+1]_1}^{j+r-1}  K_l,   & \text{ if } v\in \mathring{K_j}.
\end{cases}
\end{equation}
Thus, if we consider any 3 balls $B_1$, $B_2,$ and $B_3$ with common intersection, and  $B_p$ is the one with the largest index $l_{\rm max}$ and $B_q$ is the one with the smallest $l_{\rm min}$, for the complete graphs $K_l$ appearing in \eqref{ballcompl}, then
$$
B_1\cup B_2\cup B_3=B_p\cup B_q.
$$
Therefore, by a simple inductive argument, we deduce that $\mathcal O(\oplus K_n)=2$, which is (ii). Using \eqref{ballcompl} again, it is a straightforward calculation to show that, if $v\in B_j$, then we have
$$
|B(v,r)|\in\begin{cases}
\displaystyle \Big[\frac{jr}2,6jr\Big],& \text{ if  }  j\ge r\ge 1, \\[.3cm]
\displaystyle \Big[\frac{r^2}2,6r^2\Big], & \text{ if } 1\le j<r.
\end{cases}
$$ 
Thus, 
$$
\frac{|B(v,2r)|}{|B(v,r)|}\le 48,
$$
which proves (iii).  Remark~\ref{doubECP} and (iii) prove also (iv). 

Finally, the boundedness in (v) follows from \eqref{eq:weak11} and (ii). This finishes the proof.
\end{proof}

\subsection{The upwards shift direct sum of complete graphs: $L_\infty\oplus K_n$}
Let us consider the following variation of  $\oplus K_n$:
$$
V_{L_\infty\oplus K_n}=\{(m,n)\in \mathbb N^2: m\geq2,\,0\leq n\leq m\},
$$
with adjacent vertices as follows
$$
(m_1,n_1)\sim_{L_\infty\oplus K_n} (m_2,n_2) \Leftrightarrow \left\{
                                          \begin{array}{l}
                                            |m_1-m_2|=1,\, n_1=n_2=0;\,\mathrm{or},\\
                                            m_1=m_2,\, n_1=0,\,n_2=1;\,\mathrm{or},\\
                                            m_1=m_2,\, n_2=0,\,n_1=1;\,\mathrm{or},\\                                         
                                            m_1=m_2,\, n_1,n_2\geq1.
                                          \end{array}
                                        \right.
$$

\begin{center}
\begin{figure}[h]
\input{LKns.tex} \\[-1.5cm]
\caption{The graph $L_\infty\oplus K_n$.}
\end{figure}
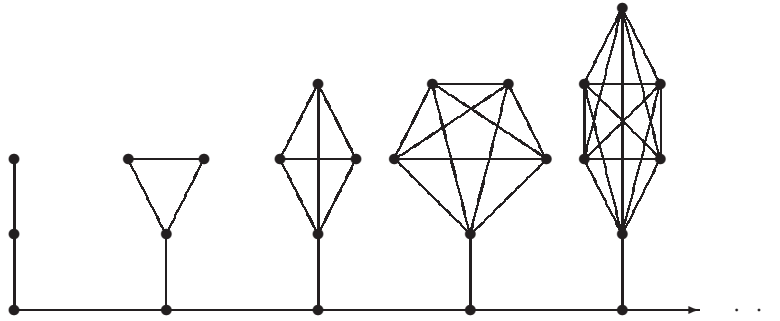
\end{center}

\begin{proposition}\label{odclk} The following properties hold for the graph $L_\infty\oplus K_n$:
\begin{enumerate}[\rm (i)]
\item $\Delta_{L_\infty\oplus K_n}=\infty$.

\item $\mathcal O(L_\infty\oplus  K_n)=5$.

\item $\mathcal D_2(L_\infty\oplus  K_n)=\infty$. 

\item$\mathcal C(L_\infty\oplus  K_n)=\infty$. 

\item $M_{L_\infty\oplus  K_n}:L_1\rightarrow L_{1,\infty}$ is bounded, with norm not exceeding $5$.
\end{enumerate}
\end{proposition}

\begin{proof}
Since $d_{(m,1)}=m$ we clearly have $\Delta_{L_\infty\oplus  K_n}=\infty$. 

That $\mathcal O(L_\infty\oplus  K_n)\geq5$ follows from the fact that for $m\geq4$, $(m,0)$ belongs to the intersection of the balls
$$
B\big((m-2,0),2\big)\cap B\big((m-1,1),2\big)\cap B\big((m,1),1\big)\cap B\big((m+1,1),2\big)\cap B\big((m+2,0),2\big),
$$
while the union of these five balls cannot be covered by only four of them. For the converse, let us consider the following:
\medskip

\noindent
{\sl Claim:} Suppose $(B_i)_{i=1}^6$ is a collection of balls in $L_\infty\oplus  K_n$, with $\bigcap_{i=1}^6 B_i\neq\emptyset$. Then, for some $i_0\in\{1,\ldots,6\}$, we have
$$
\bigcup_{i\in\{1,\ldots,6\}} B_i=\bigcup_{i\in\{1,\ldots,6\}\backslash\{i_0\}} B_i.
$$
Once this is proved, an inductive application of this argument shows that, in fact,  $\mathcal O(L_\infty\oplus  K_n)\leq5$, proving (ii).

In order to prove the claim, for $1\leq i\leq6$, let $(m_i,n_i)$ and $r_i$ denote respectively the center and radius of $B_i$. If $r_i\leq 1$, for some $i=1,\ldots,6$, it is easy to see, by simple inspection, that for some $i_0\in\{1,\ldots,6\}$ we have
$$
\bigcup_{i\in\{1,\ldots,6\}} B_i=\bigcup_{i\in\{1,\ldots,6\}\backslash\{i_0\}} B_i.
$$
Hence, without loss of generality, we can assume that $r_i\geq 2$, for each $i=1,\ldots,6$. Observe that if for some $1\leq i,j\leq6$, it holds that $m_i=m_j$, then necessarily $B_i\subset B_j$ or $B_j\subset B_i$. Hence, we can suppose that $m_0<m_1<m_2<m_3$ where $(m_0,n_0)\in\bigcap_{i=1}^6 B_i$. For each $i=1,2,3$, let us denote 
\begin{align*}
l_i=\min\{m\geq2: (m,0)\in B_i\}\text{ and }
u_i=\max\{m\geq2: (m,0)\in B_i\}.
\end{align*}
Now, if $l_2<l_1$, then $B_1\subset B_2$. Similarly, if $u_2>u_3$, then $B_3\subset B_1$. Otherwise, we have $l_1\leq l_2$ and $u_2\leq u_3$, and in this case we get $B_2\subset B_1\cup B_3$. Therefore, the claim is proved.

In order to see (iii), note that, for $m\geq3$, we have that $|B\big((m,0),1\big)|=4$, while $|B\big((m,0),2\big)|=m+7$. Hence, $\mathcal{D}_2(L_\infty\oplus K_n)=\infty$. 

As for (iv), for $m\geq3$, we find that $d\big((m,0),(m,1)\big)=1$ and $|B\big((m,0),1\big)|=4$, while $|B\big((m,1),1\big)|=m+1$.  Hence, $\mathcal{C} (L_\infty\oplus K_n)=\infty$. 

Finally, the boundedness in (v) follows from \eqref{eq:weak11} and (ii). This finishes the proof.
\end{proof}

\subsection{The $k$-regular tree: $T_k$} We consider now the infinite regular tree of degree $k$ (for $k\geq2$). The main non-trivial  result regarding this graph is the weak-type (1,1) boundedness of $M_{T_k}$, with a norm estimate independent of $k$, \cite[Theorem~1.5]{NaorTao} (see also \cite{CoMeSe,RochTab}). In the following proposition we complete the information about the remaining properties.

\begin{center}
\begin{figure}[h]
\input{treek.tex} \\
\caption{The graph $T_k$, with $k=3$.}
\end{figure}
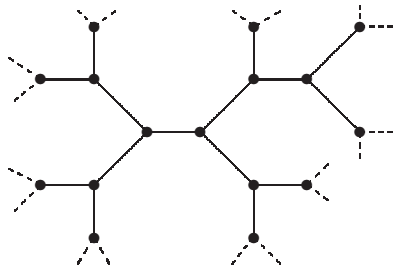
\end{center}

\begin{proposition}\label{trkdo}The following properties hold for the graph $T_k$:
\begin{enumerate}[\rm (i)]
\item $\Delta_{T_k}=k$.

\item $\mathcal O(T_k)=\infty$, for $k\geq3$.

\item $\mathcal D_2(T_k)=\infty$. 

\item$\mathcal C(T_k)=1$. 

\item \cite[Theorem~1.5]{NaorTao} $M_{T_k}:L_1\rightarrow L_{1,\infty}$ is bounded, with norm uniformly on $k$.
\end{enumerate}
\end{proposition}

\begin{proof}
Condition (i) is trivial and (v) is \cite[Theorem~1.5]{NaorTao}.  A simple calculation shows that $|B(x,r)|=(k^{r+1}-1)/(k-1)$, which proves (iv). Similarly,
$$
\frac{|B(x,2r)|}{|B(x,r)|}=\frac{k^{2r+1}-1}{k^{r+1}-1}\longrightarrow\infty,\quad\text{as }r\to\infty.
$$
Hence, $\mathcal D_2(T_k)=\infty$. 

To finish, fix a vertex $x\in T_k$ and consider the set of points $y\in S(x,r)$, the sphere of radius $r=1,2,\dots$ Then, $x$ belongs to all balls in the family  $\mathcal B_r=\big\{B(y,r)\big\}_{\{y\in S(x,r)\}}$, and for each $y\in S(x,r)$ there exists a point $x_y\in B(y,r)$ which is in no other ball of $\mathcal B_r$ (it suffices to consider $x_y$ such that $d(x_y,x)=2r$ and $y$ belongs to the unique path joining $x$ and $x_y$). Since for $k\geq3$, $|\mathcal B_r|\rightarrow\infty$, as $r\to\infty$, then $\mathcal O(T_k)=\infty$.
\end{proof}

\subsection{The infinite comb: $\Sh_\infty$}

Let $V_{\Sh_\infty}=\{(j,k):j\in\mathbb Z, k\in\mathbb N\}$. Given two vertices $(j_1,k_1)$, $(j_2,k_2)$ in $V_{\Sh_\infty}$ we define them to be adjacent according to the following
$$
(j_1,k_1)\sim_{\Sh_\infty} (j_2,k_2) \Leftrightarrow \left\{
                                          \begin{array}{c}
                                            |j_1-j_2|=1,\,\textrm{and}\, k_1=k_2=0; \textrm{ or,} \\
                                            j_1=j_2,\,\textrm{and}\, |k_1-k_2|=1.
                                          \end{array}
                                        \right.
$$

\begin{center}
\begin{figure}[h]
\input{comb.tex} \\[-.5cm]
\caption{The infinite comb graph $\Sh_\infty$.}
\end{figure}
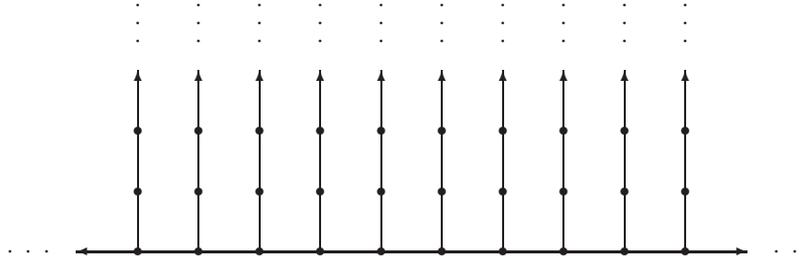
\end{center}

\begin{proposition}\label{peineks}The following properties hold for the graph $\Sh_\infty$:
\begin{enumerate}[\rm (i)]
\item $\Delta_{\Sh_\infty}=3$.

\item $\mathcal O(\Sh_\infty)=\infty$.

\item $\mathcal D_2(\Sh_\infty)=\infty$. 

\item$\mathcal C(\Sh_\infty)=\infty$. 

\item $M_{\Sh_\infty}:L_1\nrightarrow L_{1,\infty}$ is not bounded (even on Dirac deltas).
\end{enumerate}
\end{proposition}

\begin{proof}
It is clear that the degree of a vertex $(j,k)$ is
$$
d_{(j,k)}=\left\{
            \begin{array}{cc}
              3, & \textrm{if}\, k=0, \\
              2, & \textrm{otherwise.} \\
            \end{array}
          \right.
$$ This proves (i). It is easy to check that for $(j,k)\in V_{\Sh_\infty}$ and $r\in\mathbb N$ we have
$$
|B((j,k),r)|=
\left\{
\begin{array}{lc}
 2r+1, & \textrm{ if }r<k,   \\
 (r-k+1)^2+2k, & \textrm{ if }r\geq k.      
\end{array}
\right.
$$
Let $o=(0,0)$. For any $(j,k)\in V_{\Sh_\infty}$ we have that $d(o,(j,k))=|j|+k$. In particular,
$$
|B((j,k),|j|+k)|=(|j|+1)^2+2k.
$$
Now, given $\lambda\in(0,1)$ we have
\begin{align*}
|\{(j,k):|M_{\Sh_\infty}\delta_o(j,k)|>\lambda\}|&=\left|\left\{(j,k):(|j|+1)^2+2k<\frac1\lambda\right\}\right|\\
&=\sum_{j,k}\chi_{\{(|j|+1)^2+2k<\frac1\lambda\}}(j,k)=\sum_{\substack{k\in\mathbb N\\2k<\frac1\lambda-1}}\,\,\,\sum_{\substack{j\in\mathbb Z\\(|j|+1)^2<\frac1\lambda-2k}}1\\
&\geq\sum_{k=0}^{\frac1{2\lambda}-2}\left|\left\{j\in\mathbb Z:(|j|+1)^2<\frac1\lambda-2k\right\}\right|\\
&\geq\sum_{k=0}^{\frac1{2\lambda}-2}2\left(\sqrt{\frac1\lambda-2k}-2\right)\\
&\geq2\int_0^{\frac1{2\lambda}-1}\left(\sqrt{\frac1\lambda-2s}-2\right)ds\\
&\geq\frac3{2\lambda^{\frac32}}-\frac2\lambda+2.
\end{align*}
Therefore, we get that
\begin{align*}
\|M_{\Sh_\infty}\delta_o\|_{1,\infty}&=\sup_{\lambda>0}\lambda|\{(j,k):|M_{\Sh_\infty}\delta_o\|_{1,\infty}\delta_o(j,k)|>\lambda\}|\\
&\geq\sup_{\lambda\in(0,1)}\left(\frac{3}{2\lambda^{\frac12}}-2+2\lambda\right)=\infty.
\end{align*}
This proves (v). Using (v) and \eqref{eq:weak11}, we get (ii) and (iii). Also by Proposition \ref{diracdeltas}, we get (iv). 
\end{proof}

\subsection{The steplike dyadic tree: $\Upsilon_2^\infty$}

Let us consider the set of vertices
$$
V_{\Upsilon_2^\infty}=\bigcup_{n\in\mathbb N}\{(2^n,k):0\leq k\leq 2^n\}\cup\{(j,0):j\in\mathbb N, j\neq 2^n, n\in\mathbb N\}.
$$ 
Given two vertices $(j_1,k_1)$, $(j_2,k_2)$ in $V_{\Upsilon_2^\infty}$ we define them to be adjacent according to the following
$$
(j_1,k_1)\sim_{\Upsilon_2^\infty} (j_2,k_2) \Leftrightarrow \left\{
                                          \begin{array}{c}
                                            |j_1-j_2|=1,\,\textrm{and}\, k_1=k_2=0; \textrm{ or,} \\
                                            j_1=j_2,\,\textrm{and}\, |k_1-k_2|=1.
                                          \end{array}
                                        \right.
$$

\begin{center}
\begin{figure}[h]
\input{dyadicsteps.tex} \\
\caption{The graph $\Upsilon_2^\infty$.}
\end{figure}
\end{center}

\begin{proposition}\label{dyadictree}The following properties hold for the graph $\Upsilon_2^\infty$:
\begin{enumerate}[\rm (i)]
\item $\Delta_{\Upsilon_2^\infty}=3$.

\item $\mathcal O(\Upsilon_2^\infty)=\infty$.

\item $\mathcal D_2(\Upsilon_2^\infty)<\infty$. 

\item $\mathcal C(\Upsilon_2^\infty)<\infty$. 

\item $M_{\Upsilon_2^\infty}:L_1\rightarrow L_{1,\infty}$ is bounded.
\end{enumerate}
\end{proposition}

\begin{proof}
It is clear that $\Delta_{\Upsilon_2^\infty}=3$.

In order to see (ii), let us consider for $n\geq 1$ the ball $B_n$ with center $(2^n,1)$ and radius $2^n$. It is clear that $(1,0)\in\bigcap_{n\geq1} B_n$. Moreover, for each $m\geq1$ we have 
$$
(2^m,2^m)\in B_m\backslash\bigcup_{n\neq m} B_n.
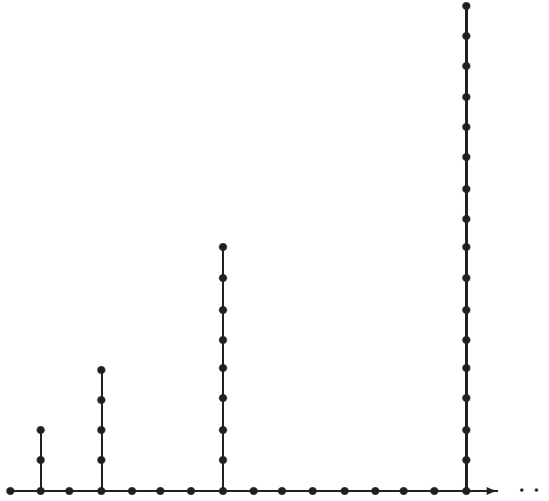
$$
This shows that $\mathcal O(\Upsilon_2^\infty)=\infty$.

The rest of the proof will follow from the following:
\medskip

\noindent
{\sl Claim:} For every $x\in V_{\Upsilon_2^\infty}$ and $r\in\mathbb N$
$$
r\leq |B(x,r)|\leq 24 r.
$$
Once this is proved, it is clear that $\mathcal D_2(\Upsilon_2^\infty)\leq 48$, $\mathcal C(\Upsilon_2^\infty)\leq 24$ and by \eqref{eq:weak11} we get $\|M_{\Upsilon_2^\infty}\|_{1,\infty}\leq 72$. 
\medskip

The proof of the claim will be split into several steps:
\begin{enumerate}[{(a)}]
\item Suppose first $x=(2^n,0)$ for some $n\in\mathbb N$. 

\begin{enumerate}
\item Let us consider $r_j=2^n-2^j$, for some $0\leq j<n$. In this case, the ball $B(x,r_j)$ contains $2r_j+1$ points of the form $(m,0)$ and can only have points of the form $(2^i,k)$, with $k\neq0$, whenever $j+1\leq i\leq n$. Thus, we have that
\begin{equation}\label{1.1}
|B((2^n,0),r_j)|\leq 2r_j+\sum_{i=j+1}^n 2^i=2r_j+(2^{n+1}-2^{j+1})=4r_j.
\end{equation}
\item Consider now $r_j=2^j-2^n$, for some $j\geq n+1$. In this case, the ball $B(x,r_j)$ consists of all the points of the form $(m,0)$ for $1\leq m\leq 2^j$ and those of the form $(2^i,k)$, with $0<k\leq 2^i$ for $i<j$. Hence,
$$
|B((2^n,0),r_j)|\leq 2^j+\sum_{i=0}^{j-1} 2^i\leq 2^{j+1}.
$$
Note that for $j\geq n+1$, we have $2^{j-1}\leq 2^j-2^n$, so we get that
\begin{equation}\label{1.2}
|B((2^n,0),r_j)|\leq 4r_j.
\end{equation}
\item Let us consider now arbitrary $r\geq1$. For $1\leq r<2^{n-1}$, we have
$$
|B((2^n,0),r)|=3r+1.
$$
For $2^{n-1}\leq r<2^n$, there is some $0\leq j< n-1$ such that $2^n-2^{j+1}<r\leq2^n-2^j$. Since for $j<n-1$ we have $2^j<2^n-2^{j+1}$, and using \eqref{1.1}, we get
\begin{align*}
|B((2^n,0),r)|&\leq|B((2^n,0),2^n-2^j)|\leq 4(2^n-2^j)\\
&= 4(2^n-2^{j+1})+2^{j+2}\leq 8 (2^n-2^{j+1})\leq 8r.
\end{align*}
Finally, for $r\geq 2^n$, there is $j\geq n+1$ such that $2^j-2^n\leq r<2^{j+1}-2^n$. In particular, for such $j$ we have $2^j\leq2(2^j-2^n)$. Hence, using \eqref{1.2}, we have
\begin{align*}
|B((2^n,0),r)|&\leq|B((2^n,0),2^{j+1}-2^n)|\leq 4(2^{j+1}-2^n)\\
&=4(2^j+2^j-2^n)\leq 12(2^j-2^n)\leq 12 r. 
\end{align*}
Therefore, for any $r\geq1$ we have
$$
|B((2^n,0),r)|\leq12r.
$$
\end{enumerate}
\item Let $x=(2^n,k)$. For $r\leq k$ we clearly have $|B((2^n,k),r)|=2r+1$. While for $r>k$, all the vertices in $B((2^n,k),r)$, except at most $r$ vertices of the form $(2^n,m)$, belong to the ball $B((2^n,0),r)$. Thus, 
$$
|B((2^n,k),r)|\leq |B((2^n,0),r)|+r\leq 13 r.
$$
\item Let $x=(j,0)$ with $j\neq 2^n$ for any $n\in\mathbb N$. In particular, there is $m\in\mathbb N$ such that $2^m<j<2^{m+1}$, and we clearly have $B((j,0),r)\subset B((2^m,0),r)\cup B((2^{m+1},0),r)$. Hence,
$$
|B((j,0),r)|\leq |B((2^m,0),r)|+|B((2^{m+1},0),r)|\leq 24 r.
$$
\end{enumerate}
\end{proof}

We summarize in Table~\ref{table1} the different results we have previously obtained. To read this list, we indicate at a given entry whether the statement on the left column implies the corresponding one on the first row.

\begin{table}[htb]
$$\left.\begin{array}
{@{\hskip-2pt}c|@{\hskip-3pt}c@{\hskip-3pt}|@{\hskip-3pt}c@{\hskip-3pt}|@{\hskip-3pt}c@{\hskip-3pt}|@{\hskip-3pt}c@{\hskip-3pt}|@{\hskip-3pt}c@{\hskip-4pt}}
\Longrightarrow & {\mathcal D_k(G)<\infty} &{\mathcal O(G)<\infty}& {\it ECP} & \Delta_G<\infty & \text{\ \  Weak-type} \\
\hline {\mathcal D_k(G)<\infty} & - & \left.\begin{array}{c}\text{No} \cr\text{Prop.~\ref{dyadictree}}\end{array}\right. & \left.\begin{array}{c}\text{Yes} \cr\text{Rem.~\ref{doubECP}}\end{array}\right. &  \left.\begin{array}{c}\text{No} \cr\text{Prop.~\ref{onodel}}\end{array}\right.  &  \left.\begin{array}{c}\text{Yes} \cr\text{\eqref{eq:weak11}}\end{array}\right.  \\
\hline { \mathcal O(G)<\infty} & \left.\begin{array}{c}\text{No} \cr\text{Prop.~\ref{odclk}}\end{array}\right.  & - &  \left.\begin{array}{c}\text{No} \cr\text{Prop.~\ref{odclk}}\end{array}\right.  &  \left.\begin{array}{c}\text{No} \cr\text{Prop.~\ref{onodel}}\end{array}\right. &  \left.\begin{array}{c}\text{Yes} \cr\text{\eqref{eq:weak11}}\end{array}\right. \\
\hline {\it ECP} & \left.\begin{array}{c}\text{No} \cr\text{Prop.~\ref{trkdo}}\end{array}\right.  &  \left.\begin{array}{c}\text{No} \cr\text{Prop.~\ref{trkdo}}\end{array}\right.  & - & \left.\begin{array}{c}\text{No} \cr\text{Prop.~\ref{onodel}}\end{array}\right.  & \left.\begin{array}{c}\text{True on deltas}\cr\text{ Prop.~\ref{diracdeltas}}\end{array}\right.  \\
\hline \Delta_G<\infty & \left.\begin{array}{c}\text{No} \cr\text{Prop.~\ref{trkdo}}\end{array}\right. & \left.\begin{array}{c}\text{No} \cr\text{Prop.~\ref{trkdo}}\end{array}\right.  &  \left.\begin{array}{c}\text{No} \cr\text{Prop.~\ref{peineks}}\end{array}\right.  & - & \left.\begin{array}{c}\text{No} \cr\text{Prop.~\ref{peineks}}\end{array}\right.   \\
\hline \text{Weak-type} & \left.\begin{array}{c}\text{No} \cr \hbox{\cite{NaorTao}}\text{ \&  Prop.~\ref{trkdo}}\end{array}\right. &  \left.\begin{array}{c}\text{No} \cr \hbox{\cite{NaorTao}}\text{ \&   Prop.~\ref{trkdo}}\end{array}\right.  &  \left.\begin{array}{c}\text{No} \cr  \text{ Prop.~\ref{odclk}}\end{array}\right.  &  \left.\begin{array}{c}\text{No} \cr \text{ Prop.~\ref{onodel}}\end{array}\right.  & -\\
\hline
\end{array}\right.
$$
\medskip

\caption{All different combinations between the main parameters and properties.}
\label{table1}
\end{table}

\bigbreak

\section{Expander type bounds for the spherical maximal function}\label{expander}

In this section we will consider the spherical maximal function
$$
M^\circ_G f(x) = \sup_{r \in\mathbb N} \frac{1}{|S(x,r)|} \sum_{y \in S(x,r)} |f(y)|,
$$
where  $S(x,r)$ denotes the sphere
$$
S(x,r) = \{ y \in G: d(x,y) = r \}.
$$
Note that if the graph $G$ is infinite, then $S(x,r)$ is never empty. 

The spherical maximal function has been considered in a different discrete setting ($\mathbb Z^d$, with euclidean metric) in \cite{MSW}. See also \cite{Ionescu} for the corresponding endpoint estimates.

Since every ball can be written as the disjoint union of spheres, we have the point-wise estimate
$$
M_Gf(x) \leq M^\circ_G f(x).
$$
Therefore, a weak-type $(1,1)$-estimates for $M^\circ_G$ implies a similar estimate for $M_G$.

We will need to control the size of the spheres of a graph, so for $r\in \N$ let
$$
S_G(r)=\sup_{x\in V_G}|S(x,r)|.
$$
In all what follows we impose that the graph satisfies $S_G(r)<\infty$ for every $r\in\mathbb N$. Notice that we can compute the degree of a vertex as $d_x=|S(x,1)|$, therefore $\Delta_G=\max_{x\in V_G} d_x=S_G(1)$. In particular, the following requires $\Delta_G<\infty$.

Motivated by the ideas in \cite{NaorTao}, for a graph $G$, we introduce the function 
$$
\mathcal E_G(r)=\sup_{\substack{A,B\subset V_G,\\ \,|A|,|B|<\infty}}\frac{1}{|A||B|}\left(\sum_{x\in B}\frac{|A\cap S(x,r)|}{|S(x,r)|}\right)^2.
$$

This function is related to the expander properties of $G$. Indeed, given finite subsets $A,B\subset V_G$, let us consider 
$$
E(A,B)=\left|\left\{(x,y)\in A\times B: d(x,y)=1\right\}\right|.
$$
This quantity can be thought of as a measure for the size of common boundary between $A$ and $B$. If $G$ were a $k$-regular expander (see \cite[Section 9.2]{AS} for details), then the number $E(A,B)$ would behave approximately as what one would expect in a random graph. Moreover, for $k$-regular $G$ we always have
$$
\mathcal E_G(1)=\frac{1}{k^2}\sup_{\substack{A,B\subset V_G,\\ \,|A|,|B|<\infty}}\frac{E(A,B)^2}{|A||B|}.
$$
In particular, for the $k$-regular infinite tree $T_k$ it is shown in \cite[Lemma 5.1]{NaorTao} that $\mathcal E_{T_k}(r)\sim 1/k^r$.

Note also that $\mathcal E_G(r)$ provides a lower estimate for the size of the spheres on the graph. Indeed, for every $x\in V_G$, if we consider $A=\{x\}$ and $B=S(x,r)$, we get
$$
|S(x,r)|\geq \mathcal E_G(r)^{-1}.
$$

\begin{theorem}\label{thm:weak11}
Let $G$ be an infinite graph such that $S_G(r)<\infty$ for every $r\in \N$.
$$
\|M^\circ_G\|_{L_1(G)\rightarrow L_{1,\infty}(G)}\lesssim\sup_{n\in\N\cup\{0\}}2^{\frac{n}{2}}\sum_{\substack{r\in \N\cup \{0\}\\ S_G(r) \geq 2^{n-1}}}\mathcal E_G(r)S_G(r)^{\frac12}.
$$
\end{theorem}

For each $r \geq 0$, let $A^\circ_r$ denote the spherical averaging
operator
$$ A^\circ_r f(x) = \frac{1}{|S(x,r)|} \sum_{y \in S(x,r)} |f(y)|.$$
Thus $M^\circ_G f(x)=\sup_{r\ge0} A^\circ_r f(x)$. Following \cite{NaorTao}, let us start with a distributional estimate on $A^\circ_r$.

\begin{lemma}\label{lem:distribution}
Let $f \in L_1(G)$, $r > 0$ and $\lambda > 0$.  Then
$$
|\left\{A^\circ_r f \geq \lambda\right\}| \lesssim \sum_{\substack{n\in \N\cup\{0\}\\ 1 \leq 2^n \leq 2S_G(r)}}
2^{\frac{3n}{2}} \mathcal E_G(r)S_G(r)^{\frac12} |\left\{ |f| \geq 2^{n-1} \lambda
\right\}|.
$$
\end{lemma}

\begin{proof}  We may take $f$ to be non-negative, and by dividing $f$ by $\lambda$ we may take $\lambda=1$.
For $n\in\mathbb N$, let us consider the set
\begin{equation*}
E_n = \left\{x\in V_G: 2^{n-1} \leq f(x) < 2^n\right\}.
\end{equation*}
Let $n(r)\in\N$ be such that $2^{n(r)}\leq S_G(r)<2^{n(r)+1}$.

Since we have
\begin{equation*}
f \leq \frac{1}{2} + \sum_{n=0}^{n(r)} 2^n \1_{E_n} + f \1_{\{f \geq 2^{n(r)}\}},
\end{equation*}
we can also bound,
\begin{equation*}
A^\circ_r f \leq \frac{1}{2} + \sum_{n=0}^{n(r)} 2^n A^\circ_r \left(\1_{E_n}\right) +A^\circ_r \left(f \1_{\{f \geq 2^{n(r)}\}}\right).
\end{equation*}

Note that
\begin{equation*}
\left\{x\in V_G:A^\circ_r \left(f \1_{\{f \geq 2^{n(r)}\}}\right) (x)\neq 0 \right\}\subset \bigcup_{f(y)\geq 2^{n(r)}}S(y,r),
\end{equation*}
so in particular we have
\begin{equation*}
\left|\left\{x\in V_G:A^\circ_r \left(f \1_{\{f \geq 2^{n(r)}\}}\right) (x)\neq 0 \right\}\right| \le S_G(r) \left|\left\{x\in V_G: f(x) \geq 2^{n(r)}\right\}\right|.
\end{equation*}

Thus, putting these together we have
\begin{align*}
\left|\left\{x\in V_G: A^\circ_r f(x) \geq 1 \right\}\right| \le &\left|\left\{x\in V_G:\sum_{n=0}^{n(r)} 2^n A^\circ_r \left(\1_{E_n}\right)(x) \geq\frac12\right\}\right|\\ 
&+ S_G(r) \left|\left\{x\in V_G: f(x) \geq 2^{n(r)} \right\}\right|.
\end{align*}

Note that if $x\in V_G$ satisfies
$$
\sum_{n=0}^{n(r)} 2^n A^\circ_r \left(\1_{E_n}\right)(x) \geq \frac{1}{2}
$$
then, for some $0\leq n\leq n(r)$, we necessarily have
$$
A^\circ_r \left(\1_{E_n}\right)(x) \geq \frac{1}{2^{n+4}}\left(\frac{2^n}{S_G(r)}\right)^{1/4}.
$$
Indeed, otherwise we would have
$$
\frac12\le  \sum_{n=0}^{n(r)} 2^n A^\circ_r\left( \1_{E_n}\right)(x)\le
\frac1{16}\sum_{n=0}^{n(r)}
\left(\frac{2^n}{S_G(r)}\right)^{1/4} \le
\frac{2^{1/4}S_G(r)^{1/4}-1}{16S_G(r)^{1/4}\left(2^{1/4}-1\right)}<\frac12,
$$
which is a contradiction.  

Now, for $0\leq n\leq n(r)$ let us consider the set
$$
F_n = \left\{x\in V_G: A^\circ_r \left(\1_{E_n}\right) (x)\geq \frac{1}{2^{n+4}}\left(\frac{2^n}{S_G(r)}\right)^{1/4} \right\},
$$
which is clearly finite and satisfies
\begin{align*}
\frac{|F_n|}{2^{n+4}}\left(\frac{2^n}{S_G(r)}\right)^{1/4}&\leq\sum_{y \in F_n} A^\circ_r \left(\1_{E_n}\right)(y)\le \sum_{y\in F_n}\frac{|E_n\cap S(y,r)|}{|S(y,r)|}\\
&\le 	\left(|E_n| |F_n|\mathcal E_G(r)\right)^{\frac12}.
\end{align*}
Hence, for $0\leq n\leq n(r)$ we have
$$
|F_n| \le 2^{\frac{3n}{2}+8} \mathcal E_G(r)S_G(r)^{\frac12} |E_n|.
$$

Finally, this estimate together with the fact that $S_G(r)^{1/2}\leq2^{\frac{3(n(r)+1)}{2}}$ and that $\mathcal E_G(r)\geq1$ (since $G$ is infinite) yield

\begin{align*}
\left| \left\{x\in V_G:A^\circ_r f(x) \geq 1 \right\}\right| &\leq \sum_{n=0}^{n(r)} |F_n| +S_G(r) \left|\left\{x\in V_G: f(x) \geq 2^{n(r)}\right\}\right|\\
& \lesssim \sum_{\substack{n\in \N\cup\{0\}\\ 1 \leq 2^n \leq 2S_G(r)}}2^{\frac{3n}{2}} \mathcal E_G(r)S_G(r)^{\frac12} \left|\left\{x\in V_G: f(x) \geq 2^{n-1}\right\}\right|.
\end{align*}
\end{proof}

\begin{proof}[Proof of Theorem~\ref{thm:weak11}]
Since $M^\circ_G f = \sup_{r \geq 0} A^\circ_r f$, Lemma~\ref{lem:distribution} implies that
\begin{align*}
\left| \left\{x\in V_G:M^\circ_G f (x)\geq \lambda\right\}\right| & \leq \sum_{r = 0}^\infty \left| \left\{ x\in V_G:A^\circ_r f (x)\geq \lambda \right\}\right|  \\
&\lesssim \sum_{r =0}^\infty \sum_{\substack{n\in \N\cup\{0\}\\ 1 \leq 2^n \leq 2S_G(r)}}\!\!\! 2^{\frac{3n}{2}} \mathcal E_G(r)S_G(r)^{\frac12} \left| \left\{ x\in V_G:|f|(x) \geq 2^{n-1}\lambda \right\}\right|\\
&=     \sum_{x \in T} \sum_{n =0}^\infty \sum_{\substack{r\in \N\cup \{0\}\\ S_G(r) \geq 2^{n-1}}}\mathcal E_G(r)S_G(r)^{\frac12} 2^{\frac{3n}{2}} \1_{\{|f(x)| \geq 2^{n-1}
\lambda\}}\\
&\lesssim \bigg(\sup_{n\in\N\cup\{0\}}2^{\frac{n}{2}}\!\!\! \sum_{\substack{r\in \N\cup \{0\}\\ S_G(r) \geq 2^{n-1}}}\!\!\mathcal E_G(r)S_G(r)^{\frac12}\bigg)\sum_{x \in T} \sum_{n =0}^\infty 2^n \1_{\{|f(x)| \geq 2^{n-1} \lambda\}} \\
&\lesssim \bigg(\sup_{n\in\N\cup\{0\}}2^{\frac{n}{2}}\sum_{\substack{r\in \N\cup \{0\}\\ S_G(r) \geq 2^{n-1}}}\mathcal E_G(r)S_G(r)^{\frac12}\bigg)\sum_{x \in T} \frac{1}{\lambda} |f(x)|,
\end{align*}
as desired.  The proof of Theorem \ref{thm:weak11} is complete. \end{proof}

\end{document}

%% file: Kns.tex
\ifx\JPicScale\undefined\def\JPicScale{0.8}\fi
\unitlength \JPicScale mm
\begin{picture}(120,45)(0,0)

\linethickness{0.15mm}
\put(20,30){\line(1,0){10}}
\put(20,30){\circle*{1.25}}

\linethickness{0.15mm}
\multiput(30,30)(0.12,0.12){83}{\line(1,0){0.12}}
\put(30,30){\circle*{1.25}}

\linethickness{0.15mm}
\multiput(40,40)(0.12,-0.12){83}{\line(1,0){0.12}}
\put(40,40){\circle*{1.25}}

\linethickness{0.15mm}
\put(30,30){\line(1,0){20}}

\linethickness{0.15mm}
\multiput(50,30)(0.12,0.12){83}{\line(1,0){0.12}}
\put(50,30){\circle*{1.25}}

\linethickness{0.15mm}
\multiput(60,40)(0.12,-0.12){83}{\line(1,0){0.12}}
\put(60,40){\circle*{1.25}}

\linethickness{0.15mm}
\multiput(60,20)(0.12,0.12){83}{\line(1,0){0.12}}
\put(60,20){\circle*{1.25}}

\linethickness{0.15mm}
\multiput(50,30)(0.12,-0.12){83}{\line(1,0){0.12}}
\put(50,30){\circle*{1.25}}

\linethickness{0.15mm}
\put(50,30){\line(1,0){20}}


\linethickness{0.15mm}
\put(60,20){\line(0,1){20}}

\linethickness{0.15mm}
\put(70,30){\line(1,0){20}}
\put(70,30){\circle*{1.25}}

\linethickness{0.15mm}
\multiput(70,30)(0.12,-0.25){42}{\line(0,-1){0.36}}

\linethickness{0.15mm}
\put(75,20){\line(1,0){10}}
\put(75,20){\circle*{1.25}}

\linethickness{0.15mm}
\multiput(85,20)(0.12,0.25){42}{\line(0,1){0.36}}
\put(85,20){\circle*{1.25}}

\linethickness{0.15mm}
\multiput(80,40)(0.12,-0.12){83}{\line(0,-1){0.18}}
\put(80,40){\circle*{1.25}}

\linethickness{0.15mm}
\multiput(70,30)(0.12,0.12){83}{\line(0,1){0.18}}%

\linethickness{0.15mm}
\multiput(75,20)(0.12,0.49){42}{\line(0,1){0.71}}

\linethickness{0.15mm}
\multiput(80,40)(0.12,-0.49){42}{\line(0,-1){0.71}}

\linethickness{0.15mm}
\multiput(70,30)(0.12,-0.08){125}{\line(1,0){0.12}}

\linethickness{0.15mm}
\multiput(75,20)(0.12,0.08){125}{\line(1,0){0.12}}

\linethickness{0.15mm}
\multiput(90,30)(0.12,0.12){83}{\line(1,0){0.12}}
\put(90,30){\circle*{1.25}}

\linethickness{0.15mm}
\put(100,40){\line(1,0){10}}
\put(100,40){\circle*{1.25}}

\linethickness{0.15mm}
\multiput(110,40)(0.12,-0.12){83}{\line(1,0){0.12}}
\put(110,40){\circle*{1.25}}

\linethickness{0.15mm}
\multiput(110,20)(0.12,0.12){83}{\line(1,0){0.12}}
\put(110,20){\circle*{1.25}}

\linethickness{0.15mm}
\put(100,20){\line(1,0){10}}
\put(100,20){\circle*{1.25}}

\linethickness{0.15mm}
\multiput(90,30)(0.12,-0.12){83}{\line(1,0){0.12}}


\linethickness{0.15mm}
\put(90,30){\line(1,0){30}}

\linethickness{0.15mm}
\put(100,20){\line(0,1){20}}

\linethickness{0.15mm}
\multiput(100,40)(0.12,-0.24){83}{\line(0,-1){0.24}}

\linethickness{0.15mm}
\multiput(100,40)(0.24,-0.12){83}{\line(1,0){0.24}}

\linethickness{0.15mm}
\multiput(90,30)(0.24,0.12){83}{\line(1,0){0.24}}

\linethickness{0.15mm}
\multiput(100,20)(0.12,0.24){83}{\line(0,1){0.24}}

\linethickness{0.15mm}
\put(110,20){\line(0,1){20}}

\linethickness{0.15mm}
\multiput(100,20)(0.24,0.12){83}{\line(1,0){0.24}}

\linethickness{0.15mm}
\multiput(90,30)(0.24,-0.12){83}{\line(1,0){0.24}}
\put(120,30){\circle*{1.25}}

\put(125,30){\circle*{0.25}}
\put(128,30){\circle*{0.25}}
\put(131,30){\circle*{0.25}}
\end{picture}

%% file: LKns.tex
\ifx\JPicScale\undefined\def\JPicScale{1}\fi
\unitlength \JPicScale mm
\begin{picture}(100,60)(0,0)

\linethickness{0.15mm}
\put(10,20){\line(1,0){90}}
\put(100,20){\vector(1,0){0.12}}
\put(10,20){\circle*{1.25}}
\put(10,30){\circle*{1.25}}
\put(10,40){\circle*{1.25}}

\linethickness{0.15mm}
\put(10,20){\line(0,1){10}}

\linethickness{0.15mm}
\put(10,30){\line(0,1){10}}

\linethickness{0.15mm}
\put(30,20){\line(0,1){10}}
\put(30,20){\circle*{1.25}}
\put(30,30){\circle*{1.25}}
\put(25,40){\circle*{1.25}}
\put(35,40){\circle*{1.25}}

\linethickness{0.15mm}
\multiput(25,40)(0.12,-0.24){42}{\line(0,-1){0.24}}

\linethickness{0.15mm}
\put(25,40){\line(1,0){10}}

\linethickness{0.15mm}
\multiput(30,30)(0.12,0.24){42}{\line(0,1){0.24}}

\linethickness{0.15mm}
\put(50,20){\line(0,1){10}}
\put(50,20){\line(0,1){10}}
\put(50,20){\circle*{1.25}}
\put(50,30){\circle*{1.25}}
\put(50,50){\circle*{1.25}}
\put(45,40){\circle*{1.25}}
\put(55,40){\circle*{1.25}}

\linethickness{0.15mm}
\multiput(45,40)(0.12,-0.24){42}{\line(0,-1){0.24}}

\linethickness{0.15mm}
\multiput(45,40)(0.12,0.24){42}{\line(0,1){0.24}}

\linethickness{0.15mm}
\multiput(50,50)(0.12,-0.24){42}{\line(0,-1){0.24}}

\linethickness{0.15mm}
\multiput(50,30)(0.12,0.24){42}{\line(0,1){0.24}}

\linethickness{0.15mm}
\put(50,30){\line(0,1){20}}

\linethickness{0.15mm}
\put(45,40){\line(1,0){10}}

\linethickness{0.15mm}
\put(70,20){\line(0,1){10}}
\put(70,20){\circle*{1.25}}
\put(70,30){\circle*{1.25}}
\put(65,50){\circle*{1.25}}
\put(75,50){\circle*{1.25}}
\put(60,40){\circle*{1.25}}
\put(80,40){\circle*{1.25}}

\linethickness{0.15mm}
\multiput(60,40)(0.12,-0.12){83}{\line(1,0){0.12}}

\linethickness{0.15mm}
\multiput(60,40)(0.12,0.24){42}{\line(0,1){0.24}}

\linethickness{0.15mm}
\put(65,50){\line(1,0){10}}

\linethickness{0.15mm}
\multiput(75,50)(0.12,-0.24){42}{\line(0,-1){0.24}}

\linethickness{0.15mm}
\multiput(70,30)(0.12,0.12){83}{\line(1,0){0.12}}

\linethickness{0.15mm}
\put(90,20){\line(0,1){10}}
\put(90,20){\circle*{1.25}}
\put(90,30){\circle*{1.25}}
\put(85,50){\circle*{1.25}}
\put(95,50){\circle*{1.25}}
\put(85,40){\circle*{1.25}}
\put(95,40){\circle*{1.25}}
\put(90,60){\circle*{1.25}}

\linethickness{0.15mm}
\multiput(85,40)(0.12,-0.24){42}{\line(0,-1){0.24}}

\linethickness{0.15mm}
\put(85,40){\line(0,1){10}}

\linethickness{0.15mm}
\multiput(85,50)(0.12,0.24){42}{\line(0,1){0.24}}

\linethickness{0.15mm}
\multiput(90,60)(0.12,-0.24){42}{\line(0,-1){0.24}}

\linethickness{0.15mm}
\put(95,40){\line(0,1){10}}

\linethickness{0.15mm}
\multiput(90,30)(0.12,0.24){42}{\line(0,1){0.24}}

\linethickness{0.15mm}
\put(90,30){\line(0,1){30}}

\linethickness{0.15mm}
\multiput(90,60)(0.12,-0.48){42}{\line(0,-1){0.48}}

\linethickness{0.15mm}
\multiput(85,40)(0.12,0.48){42}{\line(0,1){0.48}}

\linethickness{0.15mm}
\put(85,40){\line(1,0){10}}

\linethickness{0.15mm}
\multiput(85,40)(0.12,0.12){83}{\line(1,0){0.12}}

\linethickness{0.15mm}
\multiput(85,50)(0.12,-0.12){83}{\line(1,0){0.12}}

\linethickness{0.15mm}
\put(85,50){\line(1,0){10}}

\linethickness{0.15mm}
\multiput(85,50)(0.12,-0.48){42}{\line(0,-1){0.48}}

\linethickness{0.15mm}
\multiput(90,30)(0.12,0.48){42}{\line(0,1){0.48}}

\linethickness{0.15mm}
\multiput(65,50)(0.12,-0.48){42}{\line(0,-1){0.48}}

\linethickness{0.15mm}
\multiput(65,50)(0.18,-0.12){83}{\line(1,0){0.18}}

\linethickness{0.15mm}
\put(60,40){\line(1,0){20}}

\linethickness{0.15mm}
\multiput(60,40)(0.18,0.12){83}{\line(1,0){0.18}}

\linethickness{0.15mm}
\multiput(70,30)(0.12,0.48){42}{\line(0,1){0.48}}

\put(105,20){\circle*{0.25}}
\put(108,20){\circle*{0.25}}
\put(111,20){\circle*{0.25}}

\end{picture}

%% file: treek.tex
\unitlength .7mm
\begin{picture}(72.00,49.00)(0,0)

\linethickness{0.15mm}
\put(36.00,25.00){\circle*{2.00}}

\linethickness{0.15mm}
\put(46.00,35.00){\circle*{2.00}}

\linethickness{0.15mm}
\put(46.00,15.00){\circle*{2.00}}

\linethickness{0.15mm}
\put(26.00,25.00){\circle*{2.00}}

\linethickness{0.15mm}
\put(46.00,45.00){\circle*{2.00}}

\linethickness{0.15mm}
\put(56.00,35.00){\circle*{2.00}}

\linethickness{0.15mm}
\put(66.00,25.00){\circle*{2.00}}

\linethickness{0.15mm}

\linethickness{0.15mm}
\put(66.00,45.00){\circle*{2.00}}

\linethickness{0.15mm}
\put(26.00,25.00){\line(1,0){10.00}}

\linethickness{0.15mm}
\multiput(36.00,25.00)(0.12,0.12){83}{\line(1,0){0.12}}

\linethickness{0.15mm}
\put(46.00,35.00){\line(0,1){10.00}}

\linethickness{0.15mm}
\put(46.00,35.00){\line(1,0){10.00}}

\linethickness{0.15mm}
\multiput(56.00,35.00)(0.12,0.12){83}{\line(1,0){0.12}}

\linethickness{0.15mm}
\multiput(56.00,35.00)(0.12,-0.12){83}{\line(1,0){0.12}}

\linethickness{0.15mm}
\multiput(36.00,25.00)(0.12,-0.12){83}{\line(1,0){0.12}}

\linethickness{0.15mm}
\put(16.00,35.00){\circle*{2.00}}

\linethickness{0.15mm}
\put(16.00,15.00){\circle*{2.00}}

\linethickness{0.15mm}
\put(16.00,45.00){\circle*{2.00}}

\linethickness{0.15mm}
\put(6.00,35.00){\circle*{2.00}}

\linethickness{0.15mm}
\put(6.00,15.00){\circle*{2.00}}

\linethickness{0.15mm}
\put(16.00,5.00){\circle*{2.00}}

\linethickness{0.15mm}
\put(46.00,5.00){\circle*{2.00}}

\linethickness{0.15mm}
\put(56.00,15.00){\circle*{2.00}}

\linethickness{0.15mm}
\put(46.00,15.00){\line(1,0){10.00}}

\linethickness{0.15mm}
\put(46.00,5.00){\line(0,1){10.00}}

\linethickness{0.15mm}
\put(16.00,35.00){\line(0,1){9.00}}

\linethickness{0.15mm}
\multiput(16.00,35.00)(0.12,-0.12){83}{\line(1,0){0.12}}

\linethickness{0.15mm}
\multiput(16.00,15.00)(0.12,0.12){83}{\line(1,0){0.12}}

\linethickness{0.15mm}
\put(6.00,15.00){\line(1,0){10.00}}

\linethickness{0.15mm}
\put(16.00,5.00){\line(0,1){10.00}}

\linethickness{0.15mm}
\put(6.00,35.00){\line(1,0){1.00}}

\linethickness{0.15mm}
\put(6.00,35.00){\line(1,0){10.00}}

\linethickness{0.15mm}
\multiput(42.00,0.00)(1.14,1.43){4}{\multiput(0,0)(0.11,0.14){5}{\line(0,1){0.14}}}

\linethickness{0.15mm}
\multiput(46.00,5.00)(1.43,-1.43){4}{\multiput(0,0)(0.12,-0.12){6}{\line(1,0){0.12}}}

\linethickness{0.15mm}
\multiput(56.00,15.00)(1.60,1.60){3}{\multiput(0,0)(0.11,0.11){7}{\line(1,0){0.11}}}

\linethickness{0.15mm}
\multiput(56.00,15.00)(1.60,-1.20){3}{\multiput(0,0)(0.16,-0.12){5}{\line(1,0){0.16}}}

\linethickness{0.15mm}
\multiput(66.00,25.00)(1.71,0){4}{\line(1,0){0.86}}

\linethickness{0.15mm}
\multiput(66.00,20.00)(0,2.00){3}{\line(0,1){1.00}}

\linethickness{0.15mm}
\multiput(66.00,45.00)(1.71,0){4}{\line(1,0){0.86}}

\linethickness{0.15mm}
\multiput(66.00,45.00)(0,1.60){3}{\line(0,1){0.80}}

\linethickness{0.15mm}
\multiput(46.00,45.00)(2.00,1.20){3}{\multiput(0,0)(0.20,0.12){5}{\line(1,0){0.20}}}

\linethickness{0.15mm}
\multiput(42.00,48.00)(1.60,-1.20){3}{\multiput(0,0)(0.16,-0.12){5}{\line(1,0){0.16}}}

\linethickness{0.15mm}
\multiput(16.00,45.00)(1.60,1.20){3}{\multiput(0,0)(0.16,0.12){5}{\line(1,0){0.16}}}

\linethickness{0.15mm}
\multiput(13.00,48.00)(1.20,-1.20){3}{\multiput(0,0)(0.12,-0.12){5}{\line(1,0){0.12}}}

\linethickness{0.15mm}
\multiput(0.00,39.00)(1.71,-1.14){4}{\multiput(0,0)(0.17,-0.11){5}{\line(1,0){0.17}}}

\linethickness{0.15mm}
\multiput(1.00,31.00)(1.43,1.14){4}{\multiput(0,0)(0.14,0.11){5}{\line(1,0){0.14}}}

\linethickness{0.15mm}
\multiput(0.00,18.00)(1.71,-0.86){4}{\multiput(0,0)(0.21,-0.11){4}{\line(1,0){0.21}}}

\linethickness{0.15mm}
\multiput(1.00,10.00)(1.43,1.43){4}{\multiput(0,0)(0.12,0.12){6}{\line(1,0){0.12}}}

\linethickness{0.15mm}
\multiput(13.00,0.00)(1.20,2.00){3}{\multiput(0,0)(0.12,0.20){5}{\line(0,1){0.20}}}

\linethickness{0.15mm}
\multiput(16.00,5.00)(1.20,-2.00){3}{\multiput(0,0)(0.12,-0.20){5}{\line(0,-1){0.20}}}

\end{picture}

%% file: comb.tex
\ifx\JPicScale\undefined\def\JPicScale{0.8}\fi
\unitlength \JPicScale mm
\begin{picture}(145,55)(0,20)

\linethickness{0.15mm}
\put(20,30){\line(1,0){110}}
\put(130,30){\vector(1,0){0.12}}
\put(20,30){\vector(-1,0){0.12}}
\put(15,30){\circle*{0.25}}
\put(12,30){\circle*{0.25}}
\put(9,30){\circle*{0.25}}
\put(135,30){\circle*{0.25}}
\put(138,30){\circle*{0.25}}
\put(141,30){\circle*{0.25}}

\linethickness{0.15mm}
\put(40,30){\line(0,1){30}}
\put(40,60){\vector(0,1){0.12}}
\put(40,30){\circle*{1.25}}
\put(40,40){\circle*{1.25}}
\put(40,50){\circle*{1.25}}
\put(40,65){\circle*{0.25}}
\put(40,68){\circle*{0.25}}
\put(40,71){\circle*{0.25}}

\linethickness{0.15mm}
\put(60,30){\line(0,1){30}}
\put(60,60){\vector(0,1){0.12}}
\put(60,30){\circle*{1.25}}
\put(60,40){\circle*{1.25}}
\put(60,50){\circle*{1.25}}
\put(60,65){\circle*{0.25}}
\put(60,68){\circle*{0.25}}
\put(60,71){\circle*{0.25}}

\linethickness{0.15mm}
\put(80,30){\line(0,1){30}}
\put(80,60){\vector(0,1){0.12}}
\put(80,30){\circle*{1.25}}
\put(80,40){\circle*{1.25}}
\put(80,50){\circle*{1.25}}
\put(80,65){\circle*{0.25}}
\put(80,68){\circle*{0.25}}
\put(80,71){\circle*{0.25}}

\linethickness{0.15mm}
\put(100,30){\line(0,1){30}}
\put(100,60){\vector(0,1){0.12}}
\put(100,30){\circle*{1.25}}
\put(100,40){\circle*{1.25}}
\put(100,50){\circle*{1.25}}
\put(100,65){\circle*{0.25}}
\put(100,68){\circle*{0.25}}
\put(100,71){\circle*{0.25}}

\linethickness{0.15mm}
\put(120,30){\line(0,1){30}}
\put(120,60){\vector(0,1){0.12}}
\put(120,30){\circle*{1.25}}
\put(120,40){\circle*{1.25}}
\put(120,50){\circle*{1.25}}
\put(120,65){\circle*{0.25}}
\put(120,68){\circle*{0.25}}
\put(120,71){\circle*{0.25}}

\linethickness{0.15mm}
\put(50,30){\line(0,1){30}}
\put(50,60){\vector(0,1){0.12}}
\put(50,30){\circle*{1.25}}
\put(50,40){\circle*{1.25}}
\put(50,50){\circle*{1.25}}
\put(50,65){\circle*{0.25}}
\put(50,68){\circle*{0.25}}
\put(50,71){\circle*{0.25}}

\linethickness{0.15mm}
\put(70,30){\line(0,1){30}}
\put(70,60){\vector(0,1){0.12}}
\put(70,30){\circle*{1.25}}
\put(70,40){\circle*{1.25}}
\put(70,50){\circle*{1.25}}
\put(70,65){\circle*{0.25}}
\put(70,68){\circle*{0.25}}
\put(70,71){\circle*{0.25}}

\linethickness{0.15mm}
\put(90,30){\line(0,1){30}}
\put(90,60){\vector(0,1){0.12}}
\put(90,30){\circle*{1.25}}
\put(90,40){\circle*{1.25}}
\put(90,50){\circle*{1.25}}
\put(90,65){\circle*{0.25}}
\put(90,68){\circle*{0.25}}
\put(90,71){\circle*{0.25}}

\linethickness{0.15mm}
\put(110,30){\line(0,1){30}}
\put(110,60){\vector(0,1){0.12}}
\put(110,30){\circle*{1.25}}
\put(110,40){\circle*{1.25}}
\put(110,50){\circle*{1.25}}
\put(110,65){\circle*{0.25}}
\put(110,68){\circle*{0.25}}
\put(110,71){\circle*{0.25}}

\linethickness{0.15mm}
\put(30,30){\line(0,1){30}}
\put(30,60){\vector(0,1){0.12}}
\put(30,30){\circle*{1.25}}
\put(30,40){\circle*{1.25}}
\put(30,50){\circle*{1.25}}
\put(30,65){\circle*{0.25}}
\put(30,68){\circle*{0.25}}
\put(30,71){\circle*{0.25}}



































































\end{picture}

%% file: dyadicsteps.tex
\unitlength 0.4mm
\begin{picture}(168.13,166.25)(0,0)

\linethickness{0.15mm}
\put(1.25,9.38){\circle*{2.50}}

\linethickness{0.15mm}
\put(11.25,9.38){\circle*{2.50}}

\linethickness{0.15mm}
\put(20.63,9.38){\circle*{2.50}}

\linethickness{0.15mm}
\put(31.25,9.38){\circle*{2.50}}

\linethickness{0.15mm}
\put(41.25,9.38){\circle*{2.50}}

\linethickness{0.15mm}
\put(50.63,9.38){\circle*{2.50}}

\linethickness{0.15mm}
\put(60.63,9.38){\circle*{2.50}}

\linethickness{0.15mm}
\put(71.25,9.38){\circle*{2.50}}

\linethickness{0.15mm}
\put(81.25,9.38){\circle*{2.50}}

\linethickness{0.15mm}
\put(90.63,9.38){\circle*{2.50}}

\linethickness{0.15mm}
\put(100.63,9.38){\circle*{2.50}}

\linethickness{0.15mm}
\put(111.25,9.38){\circle*{2.50}}

\linethickness{0.15mm}
\put(121.25,9.38){\circle*{2.50}}

\linethickness{0.15mm}
\put(130.63,9.38){\circle*{2.50}}

\linethickness{0.15mm}
\put(140.63,9.38){\circle*{2.50}}

\linethickness{0.15mm}
\put(151.25,9.38){\circle*{2.50}}

\linethickness{0.15mm}
\put(11.25,19.38){\circle*{2.50}}

\linethickness{0.15mm}
\put(11.25,29.38){\circle*{2.50}}

\linethickness{0.15mm}
\put(31.25,19.38){\circle*{2.50}}

\linethickness{0.15mm}
\put(31.25,29.38){\circle*{2.50}}

\linethickness{0.15mm}
\put(31.25,39.38){\circle*{2.50}}

\linethickness{0.15mm}
\put(31.25,49.38){\circle*{2.50}}

\linethickness{0.15mm}
\put(71.25,19.38){\circle*{2.50}}

\linethickness{0.15mm}
\put(71.25,29.38){\circle*{2.50}}

\linethickness{0.15mm}
\put(71.25,40.00){\circle*{2.50}}

\linethickness{0.15mm}
\put(71.25,50.00){\circle*{2.50}}

\linethickness{0.15mm}
\put(71.25,59.38){\circle*{2.50}}

\linethickness{0.15mm}
\put(71.25,69.38){\circle*{2.50}}

\linethickness{0.15mm}
\put(71.25,80.00){\circle*{2.50}}

\linethickness{0.15mm}
\put(71.25,90.00){\circle*{2.50}}

\linethickness{0.15mm}
\put(151.25,19.38){\circle*{2.50}}

\linethickness{0.15mm}
\put(151.25,29.38){\circle*{2.50}}

\linethickness{0.15mm}
\put(151.25,40.00){\circle*{2.50}}

\linethickness{0.15mm}
\put(151.25,50.00){\circle*{2.50}}

\linethickness{0.15mm}
\put(151.25,59.37){\circle*{2.50}}

\linethickness{0.15mm}
\put(151.25,69.37){\circle*{2.50}}

\linethickness{0.15mm}
\put(151.25,80.00){\circle*{2.50}}

\linethickness{0.15mm}
\put(151.25,90.00){\circle*{2.50}}

\linethickness{0.15mm}
\put(151.25,99.38){\circle*{2.50}}

\linethickness{0.15mm}
\put(151.25,109.38){\circle*{2.50}}

\linethickness{0.15mm}
\put(151.25,120.00){\circle*{2.50}}

\linethickness{0.15mm}
\put(151.25,130.00){\circle*{2.50}}

\linethickness{0.15mm}
\put(151.25,140.00){\circle*{2.50}}

\linethickness{0.15mm}
\put(151.25,150.00){\circle*{2.50}}

\linethickness{0.15mm}
\put(151.25,160.00){\circle*{2.50}}

\linethickness{0.15mm}
\put(151.25,170.00){\circle*{2.50}}

\linethickness{0.15mm}
\put(1.25,9.38){\line(1,0){160.00}}
\put(161.25,9.38){\vector(1,0){0.12}}

\linethickness{0.15mm}
\put(151.25,8.75){\line(0,1){161.25}}

\linethickness{0.15mm}
\put(71.25,8.75){\line(0,1){81.25}}

\linethickness{0.15mm}
\put(31.25,8.75){\line(0,1){40.63}}

\linethickness{0.15mm}
\put(11.25,8.75){\line(0,1){20.00}}


\put(175,9.38){\makebox(0,0)[cc]{$\cdots$}}

\end{picture}